\documentclass[a4paper,12pt]{amsart}
\usepackage{amsmath}
\usepackage{amssymb}
\usepackage{xcolor}
\usepackage{amsthm}
\usepackage{amsfonts}
\usepackage[francais, english]{babel}

\DeclareMathSymbol{\subsetneqq}{\mathbin}{AMSb}{36}

\newcommand{\R}{\mathbb{R}}

\newcommand{\N}{\mathbb{N}}

\newcommand{\C}{\mathbb{C}}

\newcommand{\beq}{\begin{eqnarray}}
\newcommand{\eeq}{\end{eqnarray}}
\newcommand{\bq}{\begin{equation}}
\newcommand{\eq}{\end{equation}}
\newcommand{\beqn}{\begin{eqnarray*}}
\newcommand{\eeqn}{\end{eqnarray*}}
\newcommand{\bex}{\begin{exo}}
\newcommand{\eex}{\end{exo}}
\newcommand{\ben}{\begin{enumerate}}
\newcommand{\een}{\end{enumerate}}

\newtheorem{th1}{{\bf Theorem}}[section]
\newtheorem{thm}[th1]{{\bf Theorem}}
\newtheorem{lem}[th1]{{\bf Lemma}}
\newtheorem{prop}[th1]{{\bf Proposition}}

\newtheorem{rem}[th1]{\bf Remark}

\newtheorem{defi}[th1]{\bf Definition}

\author[R. Ghanmi and T. Saanouni]{R. Ghanmi and T. Saanouni}
\address{University Tunis El Manar,
Faculty of Sciences of Tunis, LR03ES04 partial differential equations and applications, 2092 Tunis, Tunisia.}
\email{\sl Tarek.saanouni@ipeiem.rnu.tn}
\email{\sl ghanmiradhia@gmail.com}


\title[Coupled fourth-order NLS]{On focusing coupled fourth-order nonlinear Schr\"odinger equations}

\date{\today}
\begin{document}
\begin{abstract}
We investigate some focusing fourth-order coupled Schr\"odinger equations. Existence of ground state and global well-posedness are obtained. Moreover, the best constant of some Gagliardo-Niremberg inequality is studied.
\end{abstract}
\maketitle
\tableofcontents
\vspace{ 1\baselineskip}
\renewcommand{\theequation}{\thesection.\arabic{equation}}
\section{Introduction}
Consider the focusing Cauchy problem for a fourth-order Schr\"odinger system with power-type nonlinearities
\begin{equation}
\left\{
\begin{array}{ll}
i\frac{\partial }{\partial t}u_j +\Delta^2 u_j=  \displaystyle\sum_{k=1}^{m}a_{jk}|u_k|^p|u_j|^{p-2}u_j ;\\
u_j(0,.)= \psi_{j},
\label{S}
\end{array}
\right.
\end{equation}
where $u_j: \R^{N} \times \R \to \C$, for $j\in[1,m]$ and $a_{jk} =a_{kj}$ are positive real numbers.\\
The m-component coupled nonlinear Schr\"odinger system with power-type nonlinearities
\begin{eqnarray*}(CNLS)_p\quad
i\frac{\partial }{\partial t}u_j +\Delta u_j + \mu_j |u_j|^{2p-2}u_j + \displaystyle\sum_{k\neq j}^{m}\beta_{jk}|u_k|^p|u_j|^{p-2}u_j=0 ,
\end{eqnarray*}
arises in many physical problems such as nonlinear optics and Bose-Einstein condensates. This models physical systems in which the field has more than one component. In nonlinear optics \cite{ak} $u_j$ denotes the $j^{th}$ component of the beam in Kerr-like photo-refractive media. The coupling constant $\beta_{jk}$ acts to the interaction between the $j^{th}$ and the $k^{th}$ components of the beam. The $ (CNLS)_p$ system arises also in the Hartree-Fock theory for a two component Bose-Einstein condensate. 
 Readers are referred to various other works \cite{Hasegawa, Zakharov} for the derivation and applications of this system.\\
Fourth-order Schr\"odinger equations have been introduced by Karpman \cite{Karpman} and Karpman-Shagalov \cite{Karpman 1} to take into account the role of small fourth-order dispersion terms in the propagation of intense laser beams in a bulk medium with Kerr nonlinearity.\\
A solution ${\bf u}:= (u_1,...,u_m)$ to \eqref{S} formally satisfies respectively conservation of the mass and the energy
\begin{gather*}
M(u_j):= \displaystyle\int_{\R^N}|u_j(x,t)|^2\,dx = M(\psi_{j});\\
E({\bf u}(t)):= \frac{1}{2}\displaystyle \sum_{j=1}^{m}\displaystyle\int_{\R^N}|\Delta u_j|^2\,dx - \frac{1}{2p}\displaystyle \sum_{j,k=1}^{m}a_{jk}\displaystyle \int_{\R^N} |u_j(x,t)|^p |u_k(x,t)|^p\,dx = E({\bf u}(0)).
\end{gather*}
Before going further let us recall some historic facts about these problems. $(CNLS)_p$ have been extensively studied by many mathematicians in recent years \cite{ambrosetti, ambrosetti 1, bratsch, dancer, dancer 1, lin, liu, lw, montefusco, sirakov, Wei}. Most of these papers are devoted to study the elliptic systems associated to $(CNLS)_p$ and various methods have been employed to construct solution for different regimes of parameter  $\beta_{jk}.$ In \cite{ambrosetti 1, lw, montefusco, S}, the existence of ground state solutions to the associated elliptic system was proved for suitable $\beta_{jk}>0$. In \cite{ambrosetti, bw, dancer 1, lin}, the existence of multiple solutions to some two coupled Schr\"odinger equations was investigated for $\beta_{jk}<0.$ 
For fourth-order Schr\"odinger equation, the model case given by a pure power nonlinearity is of particular interest. The question of well-posedness in the energy space $H^2$ was widely investigated. We denote for $p>1$ the fourth-order Schr\"odinger problem
$$(NLS)_p\quad i\partial_t u+\Delta^2u\pm u|u|^{p-1}=0,\quad u:{\mathbb R}\times{\mathbb R}^N\rightarrow{\mathbb C}.$$
This equation satisfies a scaling invariance. Indeed, if $u$ is a solution to $(NLS)_p$ with data $u_0$, then
$ u_\lambda:=\lambda^{\frac4{p-1}}u(\lambda^4\, .\,,\lambda\, .\,)$
is a solution to $(NLS)_p$ with data $\lambda^{\frac4{p-1}}u_0(\lambda\,.\,).$
For $s_c:=\frac N2-\frac4{p-1}$, the space $\dot H^{s_c}$ whose norm is invariant under the dilatation $u\mapsto u_{\lambda}$ is relevant in this theory. When $s_c=2$ which corresponds to the energy critical case, the critical power is $p_c:=\frac{N+4}{N-4}$, $N\geq 5$. Pausader \cite{Pausader} established global well-posedness in the defocusing subcritical case, namely $1< p < p_c$. Moreover, he established global well-posedness and scattering for radial data in the defocusing critical case, namely $p=p_c$. The same result in the critical case without radial condition was obtained by Miao, Xu and Zhao \cite{Miao}, for $N\geq 9$. The focusing case was treated by the last authors in \cite{Miao 1}. They obtained similar results to those proved by Kenig and Merle \cite{Merle} in the classical Schr\"odinger case. See \cite{ts} in the case of exponential nonlinearity. In \cite{eg}, we have the existence of solutions for a system of coupled nonlinear stationary biharmonic Schr\"odinger equations.\\

In this note, we combine in some meaning the two problems $(NLS)_p$ and $(CNLS)_p.$ Thus, we have to overcome two difficulties. The first one is the presence of bilaplacian in Schr\"odinger operator and the second is the coupled nonlinearities. In a recent work, the authors proved local well-posedness of \eqref{S} in the energy space \cite{rt} in the scale of $1<p\leq\frac{N}{N - 4}$ for $4\leq N \leq 6$ and $1<p<\infty$ if $N=4$.\\

The purpose of this manuscript is two-fold. First, by obtaining existence of a ground state, global well-posedness of the system \eqref{S} is discussed via potential well method. Second, using classical variational methods \cite{Weinstein}, a sharp constant of some Gagliardo-Niremberg inequality is obtained and global existence in the mass critical case is deduced.\\

 The rest of this paper is organized as follows. The next section contains the main results and some technical tools needed in the sequel. The goal of the third section, is to study the stationary problem associated to \eqref{S}. In section four, global existence is discussed via the potential-well theory. Section five is devoted to prove the minimal embedding constant for some vector Gagliardo-Nirenberg inequality. As a consequence, global existence for small initial data is shown for the mass critical case $p = 1 + \frac{4}{N}$, in the last section.\\

We end this section with some definitions. Let the product space as
$$H:={H^2({\R^N})\times...\times H^2({\R^N})}=[H^2({\R^N})]^m$$
where $H^2(\R^N)$ is the usual Sobolev space endowed with the complete norm 
$$ \|u\|_{H^2(\R^N)} := \Big(\|u\|_{L^2(\R^N)}^2 + \|\Delta u\|_{L^2(\R^N)}^2\Big)^\frac12.$$
We denote the real numbers 
 $$p_*:=1+\frac4N\quad\mbox{ and }\quad p^*:=\left\{
\begin{array}{ll}
\frac{N}{N-4}\quad\mbox{if}\quad N>4;\\
\infty\quad\mbox{if}\quad N=4.
\end{array}
\right.$$
We mention that $C$ will denote a
constant which may vary
from line to line and if $A$ and $B$ are nonnegative real numbers, $A\lesssim B$  means that $A\leq CB$. For $1\leq r\leq\infty$ and $(s,T)\in [1,\infty)\times (0,\infty)$, we denote the Lebesgue space $L^r:=L^r({\mathbb R}^N)$ with the usual norm $\|\,.\,\|_r:=\|\,.\,\|_{L^r}$, $\|\,.\,\|:=\|\,.\,\|_2$ and 
$$\|u\|_{L_T^s(L^r)}:=\Big(\int_{0}^{T}\|u(t)\|_r^s\,dt\Big)^{\frac{1}{s}},\quad \|u\|_{L^s(L^r)}:=\Big(\int_{0}^{+\infty}\|u(t)\|_r^s\,dt\Big)^{\frac{1}{s}}.$$
For simplicity, we denote the usual Sobolev Space $W^{s,p}:=W^{s,p}({\mathbb R}^N)$ and $H^s:=W^{s,2}$. If $X$ is an abstract space $C_T(X):=C([0,T],X)$ stands for the set of continuous functions valued in $X$ and $X_{rd}$ is the set of radial elements in $X$, moreover for an eventual solution to \eqref{S}, we denote $T^*>0$ it's lifespan.
\section{Background Material}
In what follows, we give the main results and some estimates needed in the sequel.
 For ${\bf u} :=(u_1,...,u_m)\in H$, we define the action
$$S({\bf u}):= \frac{1}{2}\displaystyle\sum_{j=1}^m\| u_j\|_{H^2}^2-\frac{1}{2p}\displaystyle\sum_{j,k=1}^m a_{jk} \displaystyle\int_{\R^N}|u_j u_k|^{p}\,dx.$$
If $ \alpha, \,\beta\in \R,$ we call constraint
{\small$${2K_{\alpha,\beta}({\bf u}):=\displaystyle\sum_{j=1}^m\big((2\alpha + (N - 4 )\beta) \|\Delta u_j\|^2 + (2\alpha + N \beta) \| u_j\|^2   \big) - \frac{1}{p}\displaystyle\sum_{j,k=1}^m a_{jk}\displaystyle\int_{\R^N}(2p\alpha + N \beta)|u_j u_k|^{p}\,dx.}$$}
\begin{defi}
We say that $\Psi:=(\psi_1,...,\psi_m)$ is a ground state solution to \eqref{S} if
\begin{equation}\label{E}
\Delta^2 \psi_j + \psi_j = \displaystyle\sum_{k=1}^m a_{jk}|\psi_k|^p|\psi_j|^{p - 2}\psi_j,\quad 0\neq \Psi\in H
\end{equation}
and it minimizes the problem
\begin{equation}\label{M}
m_{\alpha,\beta}:= \inf_{0\neq {\bf u}\in H_{rd}}\{ S({\bf u})\quad\mbox{s.\,t}\quad K_{\alpha,\beta}({\bf u}) = 0\}.
\end{equation}
Moreover, $\Psi$ is called vector ground state if each component is nonzero.
\end{defi}
\subsection{Main results}
First, we obtain existence of a ground state solution to \eqref{S}.
\begin{thm}\label{th2}
Take $N\geq4$, $p_*<p<p^*$ and two real numbers $(0, 0)\, \neq \, (\alpha,\beta)\in \R_{+}^2.$ Then
\begin{enumerate}
\item $ m:=m_{\alpha,\beta}$ is nonzero and independent of $(\alpha,\beta);$
\item there is a minimizer of \eqref{M}, which is some nontrivial solution to \eqref{E}.
\end{enumerate}
\end{thm}
\begin{rem}
If $\Psi\in H$ is a solution to \eqref{E}, then $e^{it}\Psi$ is a global solution of \eqref{S} said standing wave.
\end{rem}
Second, using the potential well method, we discuss the existence of a global solution to the focusing problem \eqref{S}.
\begin{thm}\label{th3}
Take $4\leq N\leq 6$ and $ p_*< p< p^*.$ Let $\Psi \in H$ and ${\bf u}\in C_{T^*}(H)$ the maximal solution to \eqref{S}. If there exist $(0,0)\neq (\alpha,\beta)\in \R_+^2 $ and $t_0\in  [0, T^*)$ such that 
$${\bf u}(t_0)\in A_{\alpha,\beta}^+:= \{ {\bf v}\in H \quad\mbox{s. t}\quad S({\bf v})<m\quad\mbox{and}\quad K_{\alpha,\beta}({\bf v})\geq 0\},$$
then ${\bf u}$ is global.
\end{thm}
Now, we study the existence of a vector ground state. For easy computation, we make the following assumptions
$$a_{jj} = \mu_j\; \mbox{ and }\; a_{jk} = \beta \quad\mbox{for}\quad j\neq k\in[1,m].$$
Then, the system \eqref{S} takes the form
\begin{equation}\label{E1}
\Delta^2\psi_j + \psi_j = \mu_j|\psi_j|^{2p -2}\psi_j + \beta \displaystyle\sum_{k\neq j}|\psi_k|^p|\psi_j|^{p - 2}\psi_j.
\end{equation}
The next result guarantees the uniqueness of a positive vector ground state.
\begin{thm}\label{unicité}
Let $4\leq N$, $p_*< p< p^*$ and $\Psi$ a ground state of \eqref{E1}. Then
\begin{enumerate}
\item
at least two components of $\Psi$ are non zero if $\beta>0$ is large enough;
\item
$\Psi$ is a unique positive vector ground state if $\beta>0$ is small enough.
\end{enumerate}
\end{thm}
In some cases, there exists a positive vector solution whose components are constant multiples of $w$ which is \cite{Swanson} the unique radial solution of
\begin{equation}\label{w}\Delta^2w + w =  w^{2p -1}.\end{equation}
The proof of the next result follows as in \cite{Nghiem}.
\begin{thm}\label{2.4}
Let $4\leq N$, $p_*< p< p^*$ and $\beta>0$ given as in the previous Theorem. Then, the system \eqref{E1} has a positive vector solution $\Psi^*$ which can be written in terms of $w,$ provided one of the following two conditions holds
\begin{enumerate}
\item $p=2$ and $0<\beta <\min_{1\leq j\leq m}\{\mu_j\}$ or $\beta>\max_{1\leq j\leq m}\{\mu_j\} ;$
\item $1<p\neq 2$.
\end{enumerate}
\end{thm}
Third, a sharp vector-valued Gagliardo-Nirenberg inequality is studied. Denote $C_{N,p,a_{jk}}$ the best constant in the estimate
\begin{equation}\label{Nirenberg}{\small P(\psi):=\frac{1}{2p}\displaystyle\sum_{j,k=1}^m a_{jk}\displaystyle\int_{\R^N} |\psi_j|^p|\psi_k|^p\,dx\leq C\left( \displaystyle\sum_{j=1}^m\|\Delta \psi_j\|^2\right)^{\frac{(p-1)N}{4}}\left(\displaystyle\sum_{j=1}^m\| \psi_j\|^2\right)^{\frac{N - p(N-4)}{4}}.} \end{equation}
The minimal constant is determined by the equation
\begin{equation}\label{vector}
\alpha: = \frac{1}{C_{N, p, a_{jk}}} = \inf_{\psi\in H_{rd}}J(\psi),
\end{equation}
where 
\begin{equation}\label{C}
 J(\psi):=\frac{\Big(\displaystyle\sum_{j=1}^m\|\Delta\psi_j\|^2\Big)^{\frac{(p-1)N}{4}} \Big(\displaystyle\sum_{j=1}^m\|\psi_j\|^2\Big)^{\frac{N - p(N - 4)}{4}} }{P(\psi)}.
\end{equation}
\begin{thm}\label{constant}
 Let $N\geq4$ and $p_*<p<p^*$. The minimum value for \eqref{vector} is achieved and the minimizer $(\psi_1^*,...,\psi_m^*)$ can be selected such that
$$\displaystyle\sum_{j=1}^m\|\Delta\psi_j^*\|^2= 1  = \displaystyle\sum_{j=1}^m\|\psi_j^*\|^2\quad \mbox{and} \quad C_{N, p, a_{jk}}= P(\psi_1^*,...,\psi_m^*). $$
Moreover, if $\beta>0$ is sufficiently small, the minimal constant in this case is given by 
$$C_{N,p}= \min\{\mu_1,...,\mu_m\}\frac{4p \big(N - p(N - 4)\big)^{\frac{(p - 1)N - 4}{4}}}{\big(N(p - 1)\big)^{\frac{(p - 1)N }{4}} \|w\|^{2p - 2}}.
$$
\end{thm}
Finally, using the minimal constant $C_{N,p,a_{jk}}$ in the previous vector-valued Gagliardo-Nirenberg inequality, we deduce that the Cauchy problem \eqref{S} has a global solution in the mass critical case $p =p_*$ provided that the initial data is sufficiently small. 
\begin{thm}\label{p}
Let $4<N\leq6$, $p =p_*$ and $\Psi:=(\psi_1,..,\psi_m) \in H$. Then there exists a unique global solution ${\bf u}\in C(\R,H)$ of the Cauchy problem \eqref{S} so long as
$$ \displaystyle\sum_{j=1}^m\|\psi_j\|^2\leq \Big( \frac{\alpha}{2}\Big)^{\frac{N}{4}}.$$
\end{thm}
In what follows, we collect some intermediate estimates.
\subsection{Tools}
Let us recall some useful Sobolev embeddings \cite{Adams,Lions}.
\begin{prop}\label{injection}
The two first injections are continuous and the last one is compact.
\begin{enumerate}
\item $ W^{s,p}(\R^N)\hookrightarrow L^q(\R^N)$ whenever
$1<p<q<\infty, \quad s>0\quad \mbox{and}\quad \frac{1}{p}\leq \frac{1}{q} + \frac {s}{N};$
\item $W^{s,p_1}(\R^N)\hookrightarrow W^{s - N(\frac{1}{p_1} - \frac{1}{p_2}),p_2}(\R^N)$ whenever $1\leq p_1\leq p_2 \leq \infty;$
\item
for $2<p< 2 p^*,$ 
\begin{equation}\label{radial} H_{rd}^2(\R^N)\hookrightarrow L^p(\R^N).\end{equation}
\end{enumerate}
\end{prop}
Finally, we give the so-called generalized Pohozaev identity \cite{sl1}.
\begin{prop}
$\Psi \in H$ is solution to \eqref{E} if and only if $S'(\Psi)=0.$ Moreover, in such a case 
$$K_{\alpha,\beta}(\Psi)=0,\quad\mbox{for any}\quad (\alpha,\beta)\in\R^2.$$
\end{prop}
\section{The stationary problem }
The goal of this section is to prove that the elliptic problem \eqref{E} associated to \eqref{S} has a ground state which is a unique vector in some cases. Let us start with some notations. For $u_j\in H^2$ and ${\bf u} :=(u_1,...,u_m),$ we denote the actions
\begin{gather*}
S^j(u_j) := \frac{1}{2}\Big( \|\Delta u_j\|^2 +  \| u_j\|^2   \Big);\\
S({\bf u}):= \frac{1}{2}\displaystyle\sum_{j=1}^m\| u_j\|_{H^2}^2-\frac{1}{2p}\displaystyle\sum_{j,k=1}^m a_{jk} \displaystyle\int_{\R^N}|u_j u_k|^{p}\,dx.
\end{gather*}
For $\lambda,\, \alpha, \,\beta\in \R,$ we introduce the scaling
$$( u_j^\lambda)^{\alpha,\beta}:= e^{\alpha\lambda}u_j(e^{-\beta \lambda}.)$$
and the differential operator
$$ \pounds_{\alpha,\beta}:H^2\to H^2,\quad u_j\mapsto \partial_\lambda((u_j^\lambda)^{\alpha,\beta})_{|\lambda=0}.$$
We extend the previous operator as follows, if $A:H^2(\R^N)\to \R,$ then 
$$\pounds_{\alpha,\beta}A(u_j):= \partial_\lambda (A((u_j^\lambda)^{\alpha,\beta}))_{|\lambda=0}.$$
Denote also, for $\alpha,\beta\in \R$ and $u_j\in H^2,$ the constraint
$$ K_{\alpha,\beta}^j(u_j):= \pounds_{\alpha,\beta}(S(u_j))= \frac{1}{2}( 2\alpha + (N - 4)\beta)\|\Delta u_j\|^2 + \frac{1}{2}( 2\alpha + N\beta)\|u_j\|^2$$
and
\begin{eqnarray*}
K_{\alpha,\beta}({\bf u})&:= &\partial_\lambda\big(S(({\bf u}^\lambda)^{\alpha,\beta})\big )_{|\lambda = 0}\\
&=& \frac{1}{2}\displaystyle\sum_{j=1}^m\Big((2\alpha + (N - 4 )\beta) \|\Delta u_j\|^2 + (2\alpha + N \beta) \| u_j\|^2   \Big) \\&- &\frac{1}{2p}\displaystyle\sum_{j,k=1}^m a_{jk}\displaystyle\int_{\R^N}(2p\alpha + N \beta)|u_j u_k|^{p}\,dx\\
&:=&\frac{1}{2}\displaystyle\sum_{j=1}^m K_{\alpha,\beta}^Q(u_j) - \frac{1}{2p}\displaystyle\sum_{j,k=1}^m a_{jk}\displaystyle\int_{\R^N}(2p\alpha + N \beta)|u_j u_k|^{p}\,dx.
\end{eqnarray*}
Finally, we introduce the quantities
$$H_{\alpha,\beta}^j(u_j) := S^j(u_j) - \frac{1}{2\alpha + N\beta}K_{\alpha,\beta}^j(u_j)= \frac{2\beta}{2\alpha + N\beta}||\Delta u_j\|^2$$
and
\begin{eqnarray*}
H_{\alpha,\beta}({\bf u})&:=& S({\bf u}) - \frac{1}{2\alpha + N\beta}K_{\alpha,\beta}({\bf u})\\
&=& \frac{1}{2\alpha +N\beta }\Big[\displaystyle\sum_{j=1}^m 2\beta \|\Delta u_j\|^2  + \alpha (1-\frac{1}{p})\displaystyle\sum_{j,k=1}^m a_{jk}\displaystyle\int_{\R^N}|u_j u_k|^{p}\,dx\Big].
\end{eqnarray*}
\subsection{Existence of ground state}
Now, we prove Theorem \ref{th2} about existence of a ground state solution to the stationary problem \eqref{E}.
\begin{rem} 
\begin{enumerate}
\item[(i)] The proof of the Theorem \ref{th2} is based on several lemmas;
\item[(ii)]we write, for easy notation, $u_j^\lambda:= (u_j^\lambda)^{\alpha,\beta},\, K:= K_{\alpha,\beta},\, K^Q:= K_{\alpha,\beta}^Q,\, \pounds:= \pounds_{\alpha, \beta}\, \mbox{and}\, H:= H_{\alpha,\beta}.$
\end{enumerate}
\end{rem}
\begin{lem} Let $(\alpha,\beta)\in \R_+^2.$ Then
\begin{enumerate}
\item $\min \big(\pounds H({\bf u}), H({\bf u})\big)\geq 0$ for all $0 \neq {\bf u} \in H.$
\item $\lambda \mapsto H({\bf u}^\lambda)$ is increasing.
\end{enumerate}
\end{lem}
\begin{proof}
We have
$$H({\bf u})\geq H^j(u_j) = \frac{2\beta}{2\alpha + N\beta}\|\Delta u_j\|^2\geq 0.$$
 Moreover, with a direct computation
{\small\begin{eqnarray*}
\pounds H({\bf u}) &= &\pounds \big(1 - \frac{\pounds}{2\alpha + N\beta}\big)S({\bf u})\\
&=& \frac{-1}{2\alpha + N \beta}\big(\pounds - (2\alpha + (N - 4)\beta)\big)\big(\pounds - (2\alpha + N \beta)\big)S({\bf u}) +( 2\alpha + (N - 4)\beta)\big(1 - \frac{\pounds}{2\alpha + N\beta}\big)S({\bf u})\\
&=& \frac{-1}{2\alpha + N \beta}\big(\pounds - (2\alpha + (N - 4)\beta)\big)\big(\pounds - (2\alpha + N \beta)\big)S({\bf u}) +( 2\alpha + (N - 4)\beta)H({\bf u}).
\end{eqnarray*}}
Since $\big (\pounds - (2\alpha + (N - 4)\beta)\big) \|\Delta u_j\|^2 = \big(\pounds - (2\alpha + N\beta)\big)\|u_j\|^2 = 0,$ we have
$\big (\pounds - (2\alpha + (N - 4)\beta)\big)  \big(\pounds - (2\alpha + N\beta)\big)\|u_j\|_{H^2}^2 =0$ and
\begin{eqnarray*}
\pounds H({\bf u})&\geq&\frac{-1}{2\alpha + N \beta}\big(\pounds - (2\alpha + (N - 4)\beta)\big)\big(\pounds - (2\alpha + N \beta)\big)\Big(\frac{-1}{2p}\displaystyle\sum_{j,k=1}^m a_{jk}\displaystyle\int_{\R^N}|u_ju_k|^p\,dx\Big)\\
&\geq& \frac{1}{2p}\frac{2\alpha (p - 1)}{2\alpha + N\beta}\big( 2\alpha(p - 1) + 4\beta \big)\displaystyle\sum_{j,k=1}^m a_{jk}\displaystyle\int_{\R^N}|u_ju_k|^p\,dx\geq 0.
\end{eqnarray*}
The last point is a consequence of the equality $ \partial_\lambda H({\bf u}^\lambda) = \pounds H({\bf u}^\lambda).$ In fact
\begin{eqnarray*}
\partial_\lambda H({\bf u}^\lambda)& =& \partial_\lambda H(e^{\alpha\lambda}{\bf u}(e^{-\beta\lambda}.))\\
&=& H^{\prime}(e^{\alpha\lambda}{\bf u}(e^{-\beta\lambda}.))\Big(\alpha e^{\alpha\lambda}{\bf u}(e^{-\beta\lambda}.) - \beta e^{-\beta \lambda} e^{\alpha\lambda}{\bf u}^{\prime}(e^{-\beta\lambda}.)\Big)\\
&=& H^{\prime}({\bf u}^\lambda)\Big(\alpha {\bf u}^\lambda - \beta e^{-\beta \lambda} {{\bf u}^{\prime}}^{\lambda}\Big)
\end{eqnarray*}
and
\begin{eqnarray*}
\pounds H({\bf u}^\lambda) &=& \partial_\mu (H({\bf u}^\lambda)^\mu)_{|\mu = 0}\\
&=& \partial_\mu\Big (H( e^{\alpha\mu}{\bf u}^\lambda(e^{-\beta\mu}.)\Big)_{\big|\mu = 0}\\
&=& \partial_\mu\Big (H( e^{\alpha(\mu + \lambda)}{\bf u}(e^{-\beta(\mu + \lambda)}.)\Big)_{\big|\mu = 0}\\
&=& H^{\prime}\big(e^{\alpha(\mu + \lambda)}{\bf u}(e^{-\beta(\mu + \lambda)})\big)\Big(\alpha e^{\alpha(\mu + \lambda)}{\bf u}(e^{-\beta(\mu + \lambda)}) - \beta e^{\alpha(\mu + \lambda)}e^{\alpha(\mu + \lambda)}{\bf u}^{\prime}(e^{-\beta(\mu + \lambda)}.)\Big)_{\big|\mu=0}\\
&=& H^{\prime}({\bf u}^\lambda)\Big(\alpha u_j^\lambda - \beta e^{-\beta \lambda} {{\bf u}^{\prime}}^{\lambda}\Big).
\end{eqnarray*}
\end{proof}
The next intermediate result is the following.
\begin{lem} \label{K>0} Let $(0,0)\neq(\alpha,\beta)\in \R_+^2$ satisfying $(N,\alpha)\neq (4,0)$ and $0\, \neq\, (u_1^n,...,u_m^n)$ be a bounded sequence of $H$ such that
$$ \lim_n\big(\displaystyle\sum_{j=1}^m K^Q(u_j^n)\big) =0.$$
Then there exists $n_0\in \N$ such that $K(u_1^n,...,u_m^n)>0$ for all $n\geq n_0.$
\end{lem}
\begin{proof}
We have,
$$K(u_1^n,...,u_m^n) =\frac{1}{2}\displaystyle\sum_{j=1}^m K^Q(u_j^n)  - \frac{(2p\alpha + N \beta)}{2p}\displaystyle\sum_{j,k=1 }^m a_{jk}\displaystyle\int_{\R^N}|u_j^n u_k^n|^{p}\,dx.$$
Using the interpolation inequality \eqref{Nirenberg},
$$\displaystyle \sum_{j,k=1}^{m}a_{jk}\displaystyle \int_{\R^N} |u_j^n|^p |u_k^n|^p \leq C \left(\displaystyle\sum_{j=1}^{m}\|\Delta u_j^n\|^2\right)^{\frac{(p-1)N}{4}}\left(\displaystyle\sum_{j=1}^{m}\|u_j^n\|^2\right)^{\frac{N-p(N -4)}{4}}.$$
Since $p_*<p<p^*$, $\min\{(2\alpha+(N-4)\beta),2\alpha+N\beta\}>0$ and
$$K^Q(u_j^n) = \Big((2\alpha + (N - 4 )\beta) \|\Delta u_j^n\|^2 + (2\alpha + N \beta) \| u_j^n\|^2   \Big) \rightarrow 0,$$
yields
$$   \displaystyle \sum_{j,k=1}^{m}a_{jk}\displaystyle \int_{\R^N} |u_j^n|^p |u_k^n|^p = o\left( \displaystyle\sum_{j=1}^m \|\Delta u_j^n\|^2\right) = o\left(\displaystyle\sum_{j=1}^m K^Q(u_j^n)  \right) .   $$
Thus
$$K(u_1^n,...,u_m^n) \simeq\frac{1}{2}\displaystyle\sum_{j=1}^m K^Q(u_j^n)\geq 0 . $$
\end{proof}
We read an auxiliary result.
\begin{lem}\label{Lemma} 
Let $(0,0)\neq(\alpha, \beta)\in\R_+^2$ satisfying $(N,\alpha)\neq (4,0)$. Then
$$m_{\alpha,\beta}  = \inf_{0\neq\phi\in H_{rd}}\big\{H(\phi)\quad\mbox{s.\, th}\quad K(\phi)\leq 0 \big\}.$$
\end{lem}
\begin{proof} 
Denoting by $a$ the right hand side of the previous equality, 
 it is sufficient to prove that $m_{\alpha,\beta}\leq a.$ Take $\phi \in H$ such that $K(\phi)<0.$ Because $\displaystyle\lim_{\lambda\rightarrow -\infty}K^Q(\phi^\lambda)=0,$ by the previous Lemma, there exists some $\lambda<0$ such that $ K(\phi^\lambda)>0.$ With a continuity argument there exists $\lambda_0\leq0$ such that $K (\phi^{\lambda_0}) = 0,$ then since $\lambda\mapsto H(\phi^\lambda)$ is increasing, we get
$$ m_{\alpha, \beta}\leq H(\phi^{\lambda_0}) \leq H(\phi).$$
This closes the proof.
\end{proof}{}
{\bf Proof of theorem \ref{th2}}\\
\underline{First case $(N,\alpha) \neq (4,0).$}\\
Let $(\phi_n):=(\phi_1^n,...,\phi_m^n) $ be a minimising sequence, namely
\begin{equation} \label{suite}0\neq (\phi_n) \in H_{rd},\quad K(\phi_n) = 0\quad \mbox{and}\quad \lim_n H(\phi_n) = \lim_n S(\phi_n) = m.\end{equation}
$\bullet$ First step: $(\phi_n)$ is bounded in $H.$\\
First subcase $\alpha\neq 0.$ Write
{\small\begin{gather*}
\alpha \Big(\displaystyle \sum_{j=1}^m \|\phi_j^n\|_{H^2}^2  - \displaystyle\sum_{j,k=1}^m a_{jk}\displaystyle\int_{\R^N}|\phi_j^n \phi_k^n|^p\,dx\Big) = \frac{\beta}{2}\Big( 4 \displaystyle \sum_{j=1}^m \|\Delta \phi_j^n\|^2 - N \displaystyle\sum_{j=1}^m \|\phi_j^n\|_{H^2}^2
 + \frac{N}{p} \displaystyle\sum_{j,k=1}^m a_{jk}\displaystyle \int_{\R^N}|\phi_j^n\phi_k^n|^p\,dx\Big);\\
\displaystyle\sum_{j=1}^m  \| \phi_j^n\|_{H^2}^2   -  \frac{1}{p}\displaystyle\sum_{j,k=1}^m a_{jk} \displaystyle\int_{\R^N}|\phi_j^n \phi_k^n|^{p}\,dx \rightarrow 2m.
\end{gather*}}
Assume that $\beta\neq 0,$ denoting $\lambda:= \frac{\beta}{2\alpha},$ yields
{\small$$\displaystyle \sum_{j=1}^m \|\phi_j^n\|_{H^2}^2  - \displaystyle\sum_{j,k=1}^m a_{jk}\displaystyle\int_{\R^N}|\phi_j^n \phi_k^n|^p\,dx = \lambda\Big( 4 \displaystyle \sum_{j=1}^m \|\Delta \phi_j^n\|^2 - N \displaystyle\sum_{j=1}^m \|\phi_j^n\|_{H^2}^2+ \frac{N}{p} \displaystyle\sum_{j,k=1}^m a_{jk}\displaystyle \int_{\R^N}|\phi_j^n\phi_k^n|^p\,dx\Big). $$}
So the following sequences are bounded
\begin{gather*}
- 4 \lambda \displaystyle\sum_{j=1}^m\|\Delta \phi_j^n\|^2 + \displaystyle \sum_{j=1}^m \|\phi_j^n\|_{H^2}^2  - \displaystyle\sum_{j,k=1}^m a_{jk}\displaystyle\int_{\R^N}|\phi_j^n \phi_k^n|^p\,dx;\\
\displaystyle\sum_{j=1}^m \|\phi_j^n\|_{H^2}^2 - \frac{1}{p} \displaystyle\sum_{j,k=1}^m a_{jk}\displaystyle \int_{\R^N}|\phi_j^n\phi_k^n|^p\,dx. 
\end{gather*}
Thus, for any real number $a,$ the following sequence is also bounded
$$ 4 \lambda \displaystyle\sum_{j=1}^m\|\Delta \phi_j^n\|^2  + (a - 1) \displaystyle\sum_{j=1}^m \|\phi_j^n\|_{H^2}^2 + (1 - \frac{a}{p}) \displaystyle\sum_{j,k=1}^m a_{jk}\displaystyle \int_{\R^N}|\phi_j^n\phi_k^n|^p\,dx.$$
Choosing $a\in(1,p),$ it follows that $(\phi_n)$ is bounded in $H.$\\
Second subcase $\alpha=0$ and $N\geq5.$ Write
$$\sum_{j=1}^m \|\Delta \phi_j^n\|^2\lesssim H_{0,\beta}({\bf \phi^n})=\frac{1}{2\alpha +N\beta }\Big[\displaystyle\sum_{j=1}^m 2\beta \|\Delta \phi_j^n\|^2  + \alpha (1-\frac{1}{p})\displaystyle\sum_{j,k=1}^m a_{jk}\displaystyle\int_{\R^N}|\phi_j^n\phi_k^n|^{p}\,dx\Big]\leq m.$$
Assume that $\displaystyle\lim_n\displaystyle\sum_{j=1}^m \| \phi_j^n\|=\infty$. Then, taking account of the interpolation inequality \eqref{Nirenberg}, we get
$$\sum_{j=1}^m \| \phi_j^n\|^2\lesssim K^Q({\bf \phi^n})=-K^N({\bf\phi^n})\lesssim \left(\displaystyle\sum_{j=1}^{m}\|\phi_j^n\|^2\right)^{\frac{N-p(N -4)}{4}}.$$
This is a contradiction because $\frac{N-p(N -4)}{4}=p-\frac{N(p-1)}4<1$.\\
$\bullet$ Second step: the limit of $(\phi_n)$ is nonzero and $m>0.$\\
Taking account of the compact injection \eqref{radial}, we take
$$ (\phi_1^n,...,\phi_m^n) \rightharpoonup \phi= (\phi_1,...,\phi_m)\quad \mbox{in}\quad H$$
and
$$ (\phi_1^n,...,\phi_m^n) \rightarrow  (\phi_1,...,\phi_m) \quad\mbox{in}\quad (L^{2p})^{(m)}.$$
The equality $K(\phi_n) =0$ implies that
$$  \frac{2\alpha + (N-4)\beta}{2} \displaystyle\sum_{j=1}^m \|\Delta \phi_j^n\|^2 + \frac{2\alpha + N\beta}{2} \displaystyle\sum_{j=1}^m \|\phi_j^n\|^2 =\frac{2\alpha p + N\beta}{2p} \displaystyle\sum_{j,k=1}^m a_{jk}\displaystyle \int_{\R^N} |\phi_j^n \phi_k^n|^p\,dx.$$
Assume that $\phi =0$. Using H\"older inequality, then because $p>2,$
$$\|\phi_j^n\phi_k^n\|_p^p \leq \|\phi_j^n\|_{2p}^p \|\phi_k^n\|_{2p}^p \rightarrow \|\phi_j\|_{2p}^p \|\phi_k\|_{2p}^p = 0.$$
Now, by lemma \ref{K>0} yields $K(\phi_n)>0$ for large $n$. This contradiction implies that 
$$\phi \neq 0.$$
With lower semi continuity of the $H^2$ norm, we have 
\begin{eqnarray*}
0 &=& \liminf_n K(\phi_n)\\
&\geq& \frac{2\alpha + (N-4)\beta}{2} \liminf_n\displaystyle\sum_{j=1}^m \|\Delta \phi_j^n\|^2 + \frac{2\alpha + N\beta}{2} \liminf_n\displaystyle\sum_{j=1}^m \|\phi_j^n\|^2\\& -&\frac{2\alpha p + N\beta}{2p} \displaystyle\sum_{j,k=1}^m a_{jk}\displaystyle \int_{\R^N} |\phi_j \phi_k|^p\,dx\\
&\geq& K(\phi).
\end{eqnarray*}
Similarly, we have $H(\phi)\leq m.$ Moreover, thanks to Lemma \ref{Lemma}, we can assume that $K(\phi) =0$ and $S(\phi) = H(\phi)\leq m.$ So that $\phi$ is a minimizer satisfying \eqref{suite} and
$$m_{\alpha,\beta}=H(\phi) = \frac{1}{2\alpha + N\beta}\Big( 2\beta\displaystyle\sum_{j=1}^m \|\Delta\phi_j\|^2  +  \alpha(1 - \frac{1}{p})\displaystyle\sum_{j,k=1}^m a_{jk}\displaystyle\int_{\R^N} |\phi_j\phi_k|^p\,dx\Big)>0.$$
$\bullet$ Third step: the limit $\phi$ is a solution to \eqref{E}.\\
There is a Lagrange multiplier $\eta \in \R$ such that $S^\prime(\phi) = \eta K^\prime(\phi).$ Thus
$$0 = K(\phi) = \pounds S(\phi)= \langle S^\prime(\phi), \pounds(\phi)\rangle = \eta\langle K^\prime(\phi), \pounds(\phi)\rangle = \eta\pounds K(\phi) = \eta\pounds^2S(\phi). $$
With a previous computation, we have
{\small\begin{eqnarray*}
-\pounds^2 S(\phi) - (2\alpha+(N-4)\beta)(2\alpha + N\beta)S(\phi)& =& - (\pounds - (2\alpha+(N-4)\beta))(\pounds - (2\alpha + N\beta))S(\phi) \\
&=&\frac{1}{2p}2\alpha(p - 1)(2\alpha(p -1) + 4\beta)\displaystyle\sum_{j,k=1}^m a_{jk}\displaystyle\int_{\R^N}|\phi_j\phi_k|^p\,dx\\
&>0.&
\end{eqnarray*}}
Therefore $\pounds^2S(\phi)<0.$
Thus $\eta =0$ and $S^\prime(\phi) = 0.$ So, $\phi$ is a ground state and $m$ is independent of $\alpha,\, \beta.$\\
\underline{Second case $\alpha = 0$ and $N=4$.}\\
Let $(\phi_n):=(\phi_1^n,...,\phi_m^n) $ be a minimising sequence, satisfying \eqref{suite}.\\
$\bullet$ First step: $(\phi_n)$ is bounded in $H.$\\
Without loss of generality, take $\beta=1.$ We have
$$ H_{0,1}(\phi_n) = \frac{1}{2}\displaystyle\sum_{j=1}^m \|\Delta \phi_j^n\|^2;$$
$$ K_{0,1}(\phi_n) = 2\displaystyle\sum_{j=1}^m\|\phi_j^n\|^2 - \frac{2}{p}\displaystyle\sum_{j,k=1}^m a_{jk}\displaystyle \int_{\R^N} |\phi_j^n\phi_k^n|^p\,dx.$$
By \eqref{suite} via the definition of $H_{0,1},$ $(\phi_n)$ is bounded in (\.H$^2)^{(m)}.$ Now because
$$ H_{0,1}(\phi_n)= H_{0,1}(\phi_n^\lambda),\quad K_{0,1}(\phi_n^\lambda) = e^{4\lambda}K_{0,1}(\phi_n) =0,$$ 
by the scaling $\phi_n^\lambda:= \phi_n(e^{-\lambda}.),$ we may assume that $\|\phi_j^n\| =1$ for $j \in [1,m].$ Thus $(\phi_n) $ is bounded in $ H.$\\
$\bullet$ Second step: the limit of $(\phi_n)$ is nonzero and $m>0.$\\
Taking account of the compact injection $\eqref{radial},$ we take
$$ (\phi_1^n,...,\phi_m^n) \rightharpoonup (\phi_1,...,\phi_m)\quad \mbox{in}\quad H$$
and
$$ (\phi_1^n,...,\phi_m^n) \rightarrow (\phi_1,...,\phi_m) \quad\mbox{in}\quad (L^{2p})^{(m)}.$$
Now, by the fact $0 = K (\phi_n),$ we have
$$  1 = \frac{1}{p}\displaystyle\sum_{j,k=1}^m\displaystyle\int_{\R^4} a_{jk}|\phi_j^n\phi_k^n|^p\,dx .$$
Moreover, if $\phi=0,$ we have
$$\|\phi_j^n\phi_k^n\|_p^p\leq \|\phi_j^n\|_{2p}^p\|\phi_k^n\|_{2p}^p\to \|\phi_j\|_{2p}^p\|\phi_k\|_{2p}^p =0,$$
which is a contradiction. Then
$$\phi \neq 0.$$
For $0<\lambda\longrightarrow0,$ we have 
$$ \displaystyle\sum_{j,k=1}^m a_{jk}\displaystyle\int_{\R^4} |\lambda \phi_j|^p |\lambda \phi_k|^p\,dx = o(K^Q(\lambda \phi)) = \lambda^2 K^Q(\phi),$$
hence $K^Q(\lambda \phi)>0.$ Thus,
$$ K_{0,1}(\phi)<0\Rightarrow \exists \lambda\in(0,1),\;\mbox{s.th}\; K_{0,1}(\lambda \phi)=0\,\,\mbox{and}\,\, H_{0,1}(\lambda \phi)\leq H_{0,1}(\phi).$$
So, we may assume that $K(\phi) =0$ and $S(\phi) = H(\phi)\leq m.$ Then $\phi$ is a minimizer and $m = H(\phi)>0.$\\
$\bullet$ Third step: The limit $\phi$ is a solution to \eqref{E}.\\
With a lagrange multiplicator $\eta\in \R,$ we have $S^\prime(\phi) = \eta K^\prime(\phi).$ Moreover, since
$$S^\prime(\phi_j) =\Delta^2\phi_j + \phi_j - \displaystyle\sum_{k=1}^m a_{jk}|\phi_k|^p|\phi_j|^{p-2}\phi_j\quad\mbox{and}\quad K^\prime(\phi_j) =  4 \phi_j - 4 \displaystyle\sum_{k=1}^m a_{jk} |\phi_k|^p|\phi_j|^{p-2}\phi_j. $$
it follows that
$$ \Delta^2\phi_j = (4\eta - 1)\big(\phi_j - \displaystyle\sum_{k=1}^m|\phi_k|^p|\phi_j|^{p -2}\phi_j\big).$$
Since $\langle \Delta ^2\phi_j,\phi_j\rangle>0$ and 
\begin{eqnarray*}
 \displaystyle\sum_{j=1}^m\displaystyle\int_{\R^N}\big( |\phi_j|^2 - \displaystyle\sum_{k=1}^m |\phi_k|^p |\phi_j|^p\big)\,dx
&=&  K_{0,1}(\phi)  -\displaystyle\sum_{j=1}^m \|\phi_j\| ^2 + (\frac{2}{p} - 1) \displaystyle\sum_{j,k=1}^m a_{jk}\displaystyle\int_{\R^N}|\phi_k\phi_j|^p\,dx \\
&=&  -\displaystyle\sum_{j=1}^m \|\phi_j\| ^2 + (\frac{2}{p} - 1) \displaystyle\sum_{j,k=1}^m a_{jk}\displaystyle\int_{\R^N}|\phi_k\phi_j|^p\,dx                    <0.
\end{eqnarray*}
Then $(4\eta - 1)<0.$ Finally, choosing a real number $\lambda$ such that $e^{-4\lambda}(4\eta - 1) = -1$, existence of a ground state follows taking account of the equality
$$ \Delta^2(\phi_j(e^{- \lambda}.)) = e^{-4\lambda}(4\eta - 1)\Big(\phi_j(e^{-\lambda}.) + \displaystyle\sum_{k=1}^m | \phi_k(e^{-\lambda}.)|^p |\phi_j(e^{-\lambda}.)|^{p - 2}\phi_j(e^{-\lambda}.)   \Big).$$
\subsection{Existence of vector ground state}
Now, we present the proof of Theorem \ref{unicité}, which contains two parts.
\begin{enumerate}
\item The first one deals with existence of a more that one non zero component ground state for large $\beta.$\\
Take $\phi: =(\phi_1,...,\phi_m)$ such that $(0,...,\phi_{j},...,0)$ is a ground state solution to \eqref{E}. So, $\phi_{j}$ satisfies
$$\Delta^2 \phi_j + \phi_j = \mu_j \phi_j|{\phi_j}|^{2p-2}\quad\mbox{and}\quad\displaystyle\sum_{j=1}^m\|\phi_{j}\|_{H^2}^2 =\displaystyle\sum_{j=1}^m \mu_j\|\phi_{j}\|_{2p}^{2p}.$$
Moreover, by Pohozaev identity it follows that
$$\frac{N - 4}{2}\displaystyle\sum_{j=1}^m\|\Delta \phi_{j}\|^2 + \frac{N}{2}\displaystyle\sum_{j=1}^m\|\phi_{j}\|^2 = \frac{N}{2p}\displaystyle\sum_{j=1}^m \mu_j\|\phi_{j}\|_{2p}^{2p}.$$
Collecting the previous identities, we can write
\begin{equation}\label{phz}
\displaystyle\sum_{j=1}^{m}\|\phi_{j}\|^2 = \Big(1 - \frac{N}{4} + \frac{N}{4p}\Big)\displaystyle\sum_{j=1}^{m} \mu_j\|\phi_{j}\|_{2p}^{2p}.
\end{equation}
Setting, for $t>0,$ the real variable function $\gamma(t):=\big(\phi_{1}(\frac{.}{t}),...,\phi_{m}(\frac{.}{t})\big),$ compute
\begin{eqnarray*}
K_{0,1}(\gamma(t)) &=& \frac{N - 4}{2 }t^{N - 4}\displaystyle\sum_{j=1}^m\|\Delta\phi_j\|^2  + \frac{N }{2}t^{ N}\displaystyle\sum_{j=1}^m\|\phi_j\|^2 - \frac{N}{2p}t^{ N}\displaystyle\sum_{j=1}^m \mu_j\|\phi_j\|_{2p}^{2p} \\ & -& \frac{N}{2p} \beta t^{ N}\displaystyle\sum_{ 1 \leq k\neq j\leq m} \displaystyle\int_{\R^N}|\phi_j\phi_k|^p\,dx
\end{eqnarray*}
and
\begin{eqnarray*}
g(t)
&:=&S(\gamma(t))\\
 &=& \frac{1}{2}t^{N - 4}\displaystyle\sum_{j=1}^m\|\Delta\phi_j\|^2  +  \frac{1}{2}t^{ N}\displaystyle\sum_{j=1}^m\|\phi_j\|^2 - \frac{1}{2p}t^{ N}\displaystyle\sum_{j=1}^m \mu_j\|\phi_j\|_{2p}^{2p} \\
 & -& \frac{1}{2p} \beta t^{ N}\displaystyle\sum_{1\leq k\neq j\leq m} \displaystyle\int_{\R^N}|\phi_j\phi_k|^p\,dx.
\end{eqnarray*} 
Thanks to \eqref{phz}, $ g(t)<0$ for large $t.$ Then, since $g(0)=0$ The maximum of $g(t)$ for $t\geq0$ is achieved at $\bar t>0.$ Precisely $g(\bar t) = \displaystyle\max_{t\geq0} g(t).$ Moreover, 
\begin{eqnarray*}
g^{\prime}(\bar t)=0&=& {\bar t}^{N -1}\Big(\frac{N -4}{2} {\bar t}^{-4}\displaystyle\sum_{j=1}^m\|\Delta \phi_j\|^2 + \frac{N}{2}\displaystyle\sum_{j=1}^m\|\phi_j\|^2 - \frac{N}{2p}\displaystyle\sum_{j=1}^m \mu_j\|\phi_j\|_{2p}^{2p} \\&-& \frac{N}{2p}\beta\displaystyle\sum_{1\leq j\neq k\leq m}\displaystyle\int_{\R^N}|\phi_j\phi_k|^p\,dx\Big).
\end{eqnarray*} 
Then,
$${\bar t}= \Big(\frac{N - 4}{N}\Big)^{\frac{1}{4}}\frac{\left(\displaystyle\sum_{j=1}^m\|\Delta\phi_j\|^2\right)^{\frac{1}{4}}}{\left( \frac{1}{p}\displaystyle\sum_{j=1}^m \mu_j\|\phi_j\|_{2p}^{2p}   +  \frac{1}{p}\beta \displaystyle\sum_{1\leq j\neq k\leq m} \displaystyle\int_{\R^N}|\phi_j\phi_k|^p\,dx - \displaystyle\sum_{j=1}^m \|\phi_j\|^2\right)^{\frac{1}{4}}}.$$
Thus, the maximum value of $g$ is 
\begin{eqnarray*}
 g(\bar t)& =& \displaystyle\max_{t\geq0} g(t)\\
 &=&{\frac{2(N -4)^{\frac{N -4}{4}}}{N^{\frac{N}{4}}}}\frac{\left(\displaystyle\sum_{j=1}^m\|\Delta\phi_j\|^2\right)^{\frac{N}{4}}}{\left( \frac{1}{p}\displaystyle\sum_{j=1}^m \mu_j\|\phi_j\|_{2p}^{2p}   +  \frac{1}{p}\beta \displaystyle\sum_{1\leq j\neq k\leq m} \displaystyle\int_{\R^N}|\phi_j\phi_k|^p\,dx - \displaystyle\sum_{j=1}^m \|\phi_j\|^2\right)^{\frac{N - 4}{4}}}. 
\end{eqnarray*}
Now, take $\bar{{\bf u}}:=(\bar{u}_1,..,\bar{u}_m)$ a ground state to \eqref{S}, when $\beta\longrightarrow\infty$, it follows from the previous equality that
$$0<m=S(\bar  {{\bf u}})\leq S(\phi_1(\frac.{\bar t}),...,\phi_m(\frac.{\bar t}))\longrightarrow0.$$
This contradiction achieves the proof.
\item The second one guarantees the uniqueness of a positive vector solution to the system \eqref{E1} when $\beta>0$ is sufficiently small.\\
It is sufficient to apply the implicit Theorem, and some known result \cite{Swanson} about existence of a unique positive radial solution when $\beta=0$.
\end{enumerate}
\section{Global well-posedness}
This section is devoted to obtain global existence of a solution to the system \eqref{S}. We start with a classical result about stable sets under the flow of \eqref{S}. Define
$$ A_{\alpha,\beta}^+:= \{ {\bf u}\in H \,|\, S({\bf u})<m\quad \mbox{and}\quad K_{\alpha,\beta}({\bf u})\geq 0\}.$$
\begin{lem}\label{lem} For any $(0,0)\neq (\alpha,\beta)\in \R_+^2,$ the set $A_{\alpha,\beta}^+$ is invariant under the flow of \eqref{S}.
\end{lem}
\begin{proof} Let $ \Psi \in A_{\alpha,\beta}^+$ and ${\bf u} \in C_{T^*}(H)$ be the maximal solution to \eqref{S}. Assume that ${\bf u}(t_0) \not \in A_{\alpha,\beta}^+$ for some $t_0\in (0,T^*)$. Since $S({\bf u})$ is conserved, we have $K_{\alpha,\beta}({\bf u}(t_0))<0.$ So, with a continuity argument, there exists a positive time $t_1\in (0, t_0)$ such that $ K_{\alpha,\beta}({\bf u}(t_1)) = 0$ and $S({\bf u}(t_1))<m.$ This contradicts the definition of $m.$
\end{proof}{}
\begin{lem}\label{fn}
For any $(0,0)\neq(\alpha,\beta)\in\R_+^2$, the sets $A_{\alpha,\beta}^+$ and $A_{\alpha,\beta}^-$ are independent of $(\alpha,\beta)$.
\end{lem}
\begin{proof}
Let $(\alpha,\beta)$ and $(\alpha',\beta')$ in ${\mathbb R}_+^2-\{(0,0)\}$. We denote, for $\delta\geq0$, the sets 
\begin{gather*}
A_{\alpha,\beta}^{+\delta}:=\{{\bf v}\in H\quad\mbox{s. t}\quad S({\bf v})<m-\delta\quad\mbox{and}\quad K_{\alpha,\beta}({\bf v})\geq0\};\\
A_{\alpha,\beta}^{-\delta}:=\{{\bf v}\in H\quad\mbox{s. t}\quad S({\bf v})<m-\delta\quad\mbox{and}\quad K_{\alpha,\beta}({\bf v})<0\}.
\end{gather*}
By the previous result, the reunion $A_{\alpha,\beta}^{+\delta}\cup A_{\alpha,\beta}^{-\delta}$ is independent of $(\alpha,\beta)$. So, it is sufficient to prove that $A_{\alpha,\beta}^{+\delta}$ is independent of $(\alpha,\beta)$. If $S({\bf v})<m$ and $K_{\alpha,\beta}({\bf v})=0$, then $ {\bf v}=0$. So, $A_{\alpha,\beta}^{+\delta}$ is open. The rescaling ${\bf v}^\lambda:=e^{\alpha\lambda}{\bf v}(e^{-\beta\lambda}.)$ implies that a neighborhood of zero is in $A_{\alpha,\beta}^{+\delta}$. Moreover, this rescaling with $\lambda\rightarrow-\infty$ gives that $A_{\alpha,\beta}^{+\delta}$ is contracted to zero and so it is connected. Now, write $$A_{\alpha,\beta}^{+\delta}=A_{\alpha,\beta}^{+\delta}\cap( A_{\alpha',\beta'}^{+\delta}\cup A_{\alpha',\beta'}^{-\delta})=(A_{\alpha,\beta}^{+\delta}\cap A_{\alpha',\beta'}^{+\delta})\cup(A_{\alpha,\beta}^{+\delta}\cap A_{\alpha',\beta'}^{-\delta}).$$ 
Since by the definition, $A_{\alpha,\beta}^{-\delta}$ is open and $0\in A_{\alpha,\beta}^{+\delta}\cap A_{\alpha',\beta'}^{+\delta}$, using a connectivity argument, we have $A_{\alpha,\beta}^{+\delta}=A_{\alpha',\beta'}^{+\delta}$.
\end{proof}
Let us prove Theorem \ref{th3}, which is the main result of this section.\\
With a translation argument, we assume that $t_0 = 0.$ Thus, $S(\Psi)<m$ and with lemma \ref{lem}, ${\bf u}(t)\in A_{\alpha,\beta}^+$ for any $t\in [0, T^*).$ Moreover,
\begin{eqnarray*}
m&\geq& \big(S -\frac{1}{2 + N} K_{1,1}\big) ({\bf u})\\
&=& H_{1,1}({\bf u})\\
&=&\frac{1}{2 + N}\Big( 2\displaystyle\sum_{j=1}^m\|\Delta u_j\|^2  + (1 - \frac{1}{p})\displaystyle\sum_{j,k=1}^ma_{jk}\displaystyle\int_{\R^N}|u_ju_k|^p\,dx \Big)\\
&\geq& \frac{2}{2 + N}\displaystyle\sum_{j=1}^m\|\Delta u_j\|^2.
\end{eqnarray*}
Thus, ${\bf u}$ is bounded in (\.H$^2)^{(m)}.$ Preciesly
$$ \sup_{0\leq t\leq T^*}\displaystyle\sum_{j=1}^m\|\Delta u_j\|^2\leq \frac{(2 + N)m}{2}.$$
Moreover, since the $L^2$ norm is conserved, we have
$$ \sup_{0\leq t\leq T^*}\displaystyle\sum_{j=1}^m\| u_j\|_{H^2}^2 <\infty.$$
Finally, $T^* = \infty.$
\section{Gagliardo-Nirenberg inequality}
In this section we determine the best constant $C_{N,p,a_{jk}}$ in the Gagliardo-Nirenberg inequality \eqref{Nirenberg}. Precisely we prove Theorem \ref{constant}.\\
For $\psi_j\in H^2$ and $\nu,\, \mu>0$, we denote the scaling $\psi_j^{\nu,\mu}(x) = \nu\psi_j(\mu x)$ and compute 
$$ \|\psi_j^{\nu,\mu}\|^2 = \nu^2\mu^{-N}\|\psi_j\|^2,\quad \|\Delta\psi_j^{\nu,\mu}\|^2 = \nu^2\mu^{4 - N}\|\Delta\psi_j\|^2;$$
$$ \|\psi_j^{\nu,\mu}\|_{2p}^{2p} = \nu^{2p}\mu^{-N}\|\psi_j\|_{2p}^{2p},\quad \|\psi_j^{\nu,\mu}\psi_k^{\nu,\mu}\|_p^p = \nu^{2p}\mu^{-N}\|\psi_j\psi_k\|_p^p.$$
Therefore, $J(\psi_1^{\nu,\mu},...,\psi_m^{\nu,\mu}) = J(\psi_1,...,\psi_m),$ for any $(\psi_1,...,\psi_m)\in H.$ Let ${(\psi_1^s,...,\psi_m^s)}$ be a minimising sequence for \eqref{vector} and
$$ \nu_s = \frac{\Big(\displaystyle\sum_{j=1}^m\|\psi_j^s\|^2\Big)^{\frac{N - 4}{8}}}{\Big(\displaystyle\sum_{j=1}^m\|\Delta\psi_j^s\|^2\Big)^{\frac{N}{8}}},\quad\mu_s = \frac{\Big(\displaystyle\sum_{j=1}^m\|\psi_j^s\|^2\Big)^{\frac{1}{4}}}{\Big(\displaystyle\sum_{j=1}^m\|\Delta\psi_j^s\|^2\Big)^{\frac{1}{4}}}.$$
By the above scaling invariance, $\{((\psi_1^s)^{\nu_s,\mu_s},..., (\psi_m^s)^{\nu_s,\mu_s})\}$ is also a minimizing sequence. Moreover, for each $s\in\N$,
$$\displaystyle\sum_{j=1}^m\|\Delta(\psi_j^s)^{\nu_s,\mu_s}\|^2 =\displaystyle\sum_{j=1}^m\|(\psi_j^s)^{\nu_s,\mu_s}\|^2 = 1.$$
$(\psi_j^s)^{\nu_s,\mu_s}$ is also a minimizing sequence and is bounded in $H_{rd}.$ Therefore, there exist $(\psi_1^*,...,\psi_m^*) \in H_{rd} $ and a subsequence, denoted $\big((\psi_1^s)^{\nu_s,\mu_s},...,(\psi_m^s)^{\nu_s,\mu_s} \big),$ such that the weak convergence holds
$$\big((\psi_1^s)^{\nu_s, \mu_s},...,(\psi_m^s)^{\nu_s,\mu_s} \big) \rightharpoonup (\psi_1^*,...,\psi_m^*)\quad\mbox{in}\quad H_{rd}. $$
 Since $p<\frac{N}{N - 4},$ thanks to the compact Sobolev injection \eqref{radial},
$$\big((\psi_1^s)^{\nu_s, \mu_s},...,(\psi_m^s)^{\nu_s,\mu_s} \big) \rightarrow (\psi_1^*,...,\psi_m^*)\quad\mbox{in}\quad(L^{2p}(\R^N))^{(m)}.$$
 Since the $L^2$ norm is weakly lower semi-continuous,
$$\displaystyle\sum_{j=1}^m\|\Delta \psi_j^*\|^2\leq 1\quad\mbox{and}\quad \displaystyle\sum_{j=1}^m\|\psi_j^*\|^2\leq 1. $$
The strong convergence in $L^{2p}$ implies
$$P\big((\psi_1^s)^{\nu_s, \mu_s},...,(\psi_m^s)^{\nu_s,\mu_s} \big) \rightarrow P(\psi_1^*,...,\psi_m^* ).$$
Hence,
$$\alpha\leq J(\psi_1^*,...,\psi_m^*)\leq\frac{1}{P(\psi_1^*,...,\psi_m^*)} = \lim_{s\rightarrow\infty}J\big((\psi_1^s)^{\nu_s, \mu_s},...,(\psi_m^s)^{\nu_s,\mu_s} \big) =\alpha.$$
Therefore,
$$  \Big(\displaystyle\sum_{j=1}^m\|\Delta\psi_j^*\|^2\Big)^{\frac{(p-1)N}{4}} \Big(\displaystyle\sum_{j=1}^m\|\psi_j^*\|^2\Big)^{\frac{N - p(N - 4)}{4}} = 1$$
and consequently
$$\displaystyle\sum_{j=1}^m\|\Delta\psi_j^*\|^2=\displaystyle\sum_{j=1}^m\|\psi_j^*\|^2 =1. $$
Combined with weak convergence, one concludes that
$$\big((\psi_1^s)^{\nu_s, \mu_s},...,(\psi_m^s)^{\nu_s,\mu_s} \big) \rightarrow (\psi_1^*,...,\psi_m^*)\quad\mbox{in}\quad H_{rd}.  $$
Thus, $\alpha = J(\psi_1^*,...,\psi_m^*) .$ 
\begin{rem}
It follows from the previous equality that $(\psi_1^*,...,\psi_m^*)$ is a minimizer of $J$ in $H_{rd}$ and satisfies the Euler-Lagrange equation
$$ \frac{d}{d\epsilon}J(\psi_1^* + \epsilon v_1,...,\psi_m^* + \epsilon v_m)|_{\epsilon =0}=0\quad \mbox{for all}\quad  (v_1,...,v_m)\in (C_0^{\infty}(\R^N))^{(m)}.$$
Taking into account of the equalities $\displaystyle\sum_{j=1}^m\|\psi_j^*\|^2 =\displaystyle\sum_{j=1}^m\|\Delta\psi_j^*\|^2 =1, $ yields
$$ \frac{(p -1)N}{2}\Delta^2\psi_j^* + \frac{N - p(N -4)}{2}\psi_j^* = \alpha \displaystyle\sum_{k=1}^m a_{jk}|\psi_k^*|^p|\psi_j^*|^{p -2}\psi_j^*.$$
\end{rem}
Now, we prove the last part of Theorem \ref{constant}.\\
With uniqueness and scaling arguments, 
$$\psi_j(x) = \Big(\frac{4p - N(p-1)}{2\alpha\mu_j}\Big)^{\frac{1}{2p - 2}}w\Big(\Big(\frac{4p - N(p-1)}{N(p -1)}\Big)^{\frac{1}{4}}x\Big) := A_j w\Big(\Big(\frac{4p - N(p-1)}{N(p -1)}\Big)^{\frac{1}{4}}x\Big).$$
Compute
$$ J(\psi_1^*,..., \psi_m^*) = 2p f_{m,p}(\beta)\frac{\|\Delta w\|^{\frac{(p - 1)N}{2}} \|w\|^{\frac{N - p(N - 4)}{2}}}{ \|w\|_{2p}^{2p}}$$
where
$$ f_{m,p}(\beta) = \frac{\Big( \displaystyle\sum_{j =1}^mA_j^2\Big)^p}{ \displaystyle\sum_{j=1}^m \mu_j A_j^{2p} + \beta \displaystyle\sum_{1\leq j\neq k\leq m} A_j^pA_k^p}.$$
Denote $\frac{A_j}{A_1} = t_j$ for $j=2,...,m,$ then
$$ f_{m,p}(\beta) = \frac{\Big( 1 + \displaystyle\sum_{j =2}^m t_j^2\Big)^p}{\mu_1 +  \displaystyle\sum_{j=1}^m \mu_j t_j^{2p} + \beta F(t_2,...,t_m)}.$$
Therefore, for $\mu_1 = \displaystyle\min (\mu_2,...,\mu_m),$ we have
$$ f_{m,p}(0) \geq \frac{\Big( 1 + \displaystyle\sum_{j =2}^m t_j^2\Big)^p}{\mu_1\Big(1 +  \displaystyle\sum_{j=1}^m  t_j^{2p}\Big) }>\frac1{\mu_1}.$$
Pohozaev identity gives
$$ \|\Delta w\|^2 = \frac{N(p - 1)}{N - p(N - 4)}\|w\|^2\quad\mbox{and}\quad\|w\|_{2p}^{2p} = \frac{4p}{N - p(N - 4)}\|w\|^2.$$
Then
$$J(\psi_j,0) = \frac{ \big(N(p - 1)\big)^{\frac{(p - 1)N }{4}} \|w\|^{2p - 2} }{ 4p \big(N - p(N - 4)\big)^{\frac{(p - 1)N - 4}{4}}  }   . $$
Therefore, for small $\beta>0$, $J(\Psi^*)\geq J(\psi_j,0)$ and so the minimal constant is given by
$$ C_{N,p}= \min\{\mu_1,...,\mu_m\}\frac{4p \big(N - p(N - 4)\big)^{\frac{(p - 1)N - 4}{4}}}{\big(N(p - 1)\big)^{\frac{(p - 1)N }{4}} \|w\|^{2p - 2}}.  $$
\section{Global existence in the mass critical case}
Using the minimal constant $C_{N,p,a_{jk}}$ in the vector-valued Gagliardo-Nirenberg inequality, we deduce that the Cauchy problem \eqref{S} is well posed in the mass critical case $p =p_*$ provided the initial data is sufficiently small. More precisely, we prove Theorem \ref{p}.\\
Recall the energy functional by
$$ E(u_1,...,u_m) =  \frac{1}{2}\displaystyle\sum_{j=1}^m\|\Delta u_j(t)\|^2 - P(u_1(t),...,u_m(t)).$$
Using the minimal constant $\mathcal{C} = C_{N,p,a_{jk}}$ obtained above with $p= p_*,$
\begin{eqnarray*}
\displaystyle\sum_{j=1}^m\|\Delta u_j(t)\|^2 &=& 2E + 2P(u_1(t),...,u_m(t))\\
&\leq& 2E + 2\mathcal{C}\Big(\displaystyle\sum_{j=1}^m\|\Delta u_j(t)\|^2\Big)^{\frac{(p - 1)N}{4}}\Big(\displaystyle\sum_{j=1}^m\| u_j(t)\|^2\Big)^{\frac{N - p(N - 4)}{4}}\\
&\leq& 2E + 2\mathcal{C}\Big(\displaystyle\sum_{j=1}^m\|\Delta u_j(t)\|^2\Big)\Big(\displaystyle\sum_{j=1}^m\| u_j(t)\|^2\Big)^{\frac{4}{N}}.
\end{eqnarray*}
This implies that
$$ \Big( 1 - 2\mathcal{C}\Big( \displaystyle\sum_{j=1}^m\| \psi_j\|^2\Big)^{\frac{4}{N}}  \Big)\displaystyle\sum_{j=1}^m\|\Delta u_j(t)\|^2 \leq 2E.$$
Therefore, if the initial data $\displaystyle\sum_{j=1}^m\| \psi_j\|^2$ are chosen small enough, namely
$$\displaystyle\sum_{j=1}^m\|\psi_j\|^2< \Big( \frac{1}{2\mathcal{C}}\Big)^{\frac{N}{4}}, $$
the $H^2$ norm is bounded and so $T^*=\infty$.


\end{document}